\documentclass[12pt,reqno]{amsart}
\usepackage[all]{xy}
\newcommand{\norm}[1]{{\left\|{#1}\right\|}}    
   \newcommand{\scal}[1]{{\left\langle{#1}\right\rangle}}

\usepackage[english]{babel} \usepackage{latexsym, amsfonts, amsmath, amssymb, amsthm, mathrsfs}
\usepackage{color,lmodern} \usepackage{mathpazo} \usepackage[pdftex]{hyperref}
\usepackage[top=3cm, bottom=3cm, left=2.53cm, right=2.53cm]{geometry}
\numberwithin{equation}{section}  \makeatletter\@addtoreset{equation}{section}
\newtheorem {theorem}{Theorem}[section]
 \newtheorem {lemma}[theorem]{Lemma}     
     \newtheorem {remark}[theorem]{Remark}   
       \newtheorem {proposition}[theorem]{Proposition}
\newcommand{\C}{\mathbb C} \newcommand{\R}{\mathbb R} 
 	 
\newcommand{\Hq}{\mathbb H}  \newcommand{\Sq}{\mathbb S}

\newcommand{\SHyperBTransR}{\mathcal{A}^\alpha_{slice}}

\begin{document}

%================================================================================================
 \dedicatory{ \textit{Dedicated to the memory of Professor Brahim Bouya}}
%================================================================================================
%\begin{frontmatter}
%%%%%%%%%%%%%%%%%%%%%%%%%%%%%%%%%%%%%%%%%%%%%%%%%%%%%%%%%%%%%%%%%%%%%%%%%%%%%%%%%%%%%%%%%%%%%%%%
\title[]{On dual transform of fractional Hankel transform} 
	\title{On dual transform of fractional Hankel transform}
\author{Allal Ghanmi}
\address{Analysis, P.D.E. $\&$ Spectral Geometry, Lab M.I.A.-S.I., CeReMAR,	Department of Mathematics, P.O. Box 1014,  Faculty of Sciences,	Mohammed V University in Rabat, Morocco}
%	\curraddr{}
\email{allalghanmi@um5.ac.ma}
\thanks{}


%Mathematics Subject Classification (2020)
%	44A20   MSC2020: Integral transforms of special functions
%	30G35   MSC2020: Functions of hypercomplex variables and generalized variables
%	30H20  MSC2020: Bergman spaces and Fock spaces	
%	Secondary 
%	47B38  MSC2020: Linear operators on function spaces (general)


\date{}

\dedicatory{\textit{Dedicated to the memory of Professor Brahim Bouya}}

%    "Communicated by" -- provide editor's name; required.
\commby{}

\begin{abstract}
	We deal with a class of one-parameter family of integral transforms of Bargmann type arising as dual transforms of fractional Hankel transform. Their ranges are identified to be special subspaces of weighted hyperholomorphic left Hilbert spaces, generalizing the slice Bergman space of second kind. Their reproducing kernel is given by closed expression involving the $\star$-regularization of Gauss hypergeometric function. 
	We also discuss their basic properties such as their boundedness and we determinate their singular values. Moreover, we describe their  compactness and membership in $p$-Schatten classes. 
\end{abstract}

\maketitle

\section{Introduction} 
The fractional Hankel transform was introduced in \cite{Namias1980b} (see also \cite{Kerr1991}). It arises as fundamental tool in different areas of sciences, see e.g.
\cite{Krenk1982,Karp1994,SheppardLarkin1998,Bracewell1999,Uzun2015}. A quaternionic analogue has been investigated recently in \cite{ElkGhHa20}. It is realized there \`a la Bargmann by means of the hyperholomorphic second Bargmann transform \cite{ElkachGh2018} 
\begin{equation}\label{BargmannT}
[\SHyperBTransR \varphi](q)=   
\frac{1}{\left( 1-q\right)^{\alpha +1}}  \int_{0}^{+\infty}  t^\alpha \exp\left(\frac{ t }{1-q}\right) \varphi(t) dt; \, \, \alpha>-1,
\end{equation}
defined on $L^{2,\alpha}_{\Hq}(\R^+):= L^{2}_{\Hq}(\R^+,x^{\alpha}e^{-x}dx)$, the right quaternionic Hilbert space  of all $\Hq$-valued functions on the half real line $\R^+$ that are $x^{\alpha}e^{-x}dx$-square integrable.
% with respect to the inner product $$\scal{\varphi,\psi}_{\alpha}=\int_{\mathbb{R}^{+}}\overline{\varphi(x)}\psi(x) x^{\alpha}e^{-x}dx  .$$ 
Such transform defines
a unitary  isometric transformation from  $L^{2,\alpha}_{\Hq}(\R^+)$
onto the slice hyperholomorphic Bergman space 
$ %\begin{align}\label{HypBspace}
A^{2,\alpha}_{slice} :=   \mathcal{SR}(\mathbb{B}) \cap L^{2,\alpha}_{\Hq}(\mathbb{B}_I), %;\,\, \alpha >-1,
$ %\end{align} 
 where $L^{2,\alpha}_{\Hq}(\mathbb{B}_I)$ denotes the Hilbert space of quaternionic-valued functions $f$ on the unit ball $\mathbb{B}$ whose restrictions $f|_{\C_I}$ to given slice $\C_I=\R+I\R$,  with $I\in \Sq=\{q\in{\Hq};q^2=-1\}$, are square integrable with respect to the Bergman measure $d\lambda^{\alpha}_I(z) = \left(1-|z|^2 \right)^{\alpha-1} dxdy$; $z=x+Iy$, on the unit disc $\mathbb{B}_I:= \mathbb{B}\cap\C_I$,
%\begin{align}\label{BergMes}
%d\lambda^{\alpha}_I(z=x+Iy) = \left(1-|z|^2 \right)^{\alpha-1} dxdy, 
%%;\,\, \alpha >0
%\end{align}
%and endowed with scaler product 
%\begin{equation}\label{spfg}
%\scal{f,g}_{\C_I} = \int_{\C_I}\overline{f|_{\C_I}(q)} g|_{\C_I}(q) d\lambda^{\alpha}_I(q).
%\end{equation} 
while, the space $ \mathcal{SR}(\mathbb{B})$ is formed of all slice regular functions on $\mathbb{B}$,  i.e., those whose slice derivative
\begin{align*} 
( \overline{\partial_I} f)(x+Iy) := %= \overline{\partial_{s}}f|_{\C_I}(x+Iy) :=
\dfrac{1}{2}\left(\frac{\partial }{\partial x}+I\frac{\partial }{\partial y}\right)f|_{\C_I}(x+yI)
\end{align*}
 vanishes identically for every $I\in \Sq$.

In the present paper, we will consider the one--parameter (left) integral transforms  $S^{\alpha}_{y}$; $\alpha >-1$, defined on $L^{2,\alpha}_{\Hq}(\R^+)$ as the dual transform at $y\in (0,+\infty)$  of the quaternionic fractional Hankel transform $\mathcal{L}_\theta^\alpha$ defined in \cite{ElkGhHa20}. More precisely, we deal with 
\begin{align}\label{SIntT}
S^{\alpha}_{y} \varphi (q) :=   
\frac{1}{1-q}   \int_{0}^\infty 
\exp\left(-\frac{x+\theta y}{1-q}\right) 
I_\alpha\left( \frac{2\sqrt{q xy }}{1-q} \right)
\varphi (x) dx% =: \mathcal{L}_q^\alpha \varphi(y) 
, 
\end{align}
where  $I_\alpha$ stands for the modified Bessel function \cite[p. 222, (4.12.2)]{AndrewsAskeyRoy1999}
$$ I_\alpha (\xi) =  \sum_{n=0}^{\infty} \frac{1}{n! \Gamma(\alpha+n+1)} \left(  \frac{\xi}{2}\right)^{2n+\alpha} .%; \, \alpha>-1.
$$
The motivation of considering $S^{\alpha}_{y} $ lies on the observation that the limiting case $y=0$ gives rise to the
hyperholomorphic second Bargmann transform in \eqref{BargmannT}.
Accordingly, the study of these new operators is required in order to generate hyperholomorphic-like Bergman spaces.

Our first aim is to identify the null space and the range of $S^{\alpha}_{y}$ for arbitrary $y\geq 0$. We show that the range of $L^{2,\alpha}_{\Hq}(\R^+)$ by the transform $S^{\alpha}_{y}$ is contained in a reproducing kernel weighted slice hyperholomorphic 
%right Hilbert space 
with suitable weight $\omega$ on the unit ball $\mathbb{B}$, extending $A^{2,\alpha}_{slice}$.
%the one in \eqref{HypBspace}. 
  The description of $S^{\alpha}_{y}(L^{2,\alpha}_{\Hq}(\R^+))$ is given by Proposition \ref{range} and makes appeal to the zeros of the Laguerre polynomials.
  %, while the closed expression of its reproducing kernel is given in Proposition \ref{kernel}. 
We also study their boundedness (Proposition \ref{thmBound}) and compactness (Proposition \ref{corcomp}). Moreover, we determine their singular values (Proposition \ref{propsv}) and we discuss their membership in $p$-Schatten classes (Proposition \ref{thmSchayyen}).
Explicit illustration are given for a specific weight function $\omega=\omega_{\beta,\eta}$. 

%%%%%%%%%%%%%%%%%%%%%%%%%%%%%%%%%%%%%%%%%%%%%%%%%
\section{Weighted hyperholomorphic left Bergman Hilbert spaces}
%%%%%%%%%%%%%%%%%%%%%%%%%%%%%%%%%%%%%%%%%%%%%%%%%
In order to identify the range of the integral transform $S^{\alpha}_{y}$ in \eqref{SIntT} when acting on  $L^{2,\alpha}_{\Hq}(\R^+)$, we begin by examining a class of weighted hyperholomorphic left Bergman Hilbert spaces for which we provide a closed expression of their reproducing kernel in terms of a $\star$-regularization of Gauss hypergeometric function. 
Let $\omega$ be a given positive measurable mapping on $(0,1)$ such that $\omega(t) dt$ be a finite measure. We extend $\omega$ to the whole unit ball $\mathbb{B}\subset \Hq$ by taking $\widetilde{\omega}(q):= \omega (|q|^2)$.
We define the $\omega$-hyperholomorphic left Bergman Hilbert space 
$A^{2,\omega}_{slice} := \mathcal{SR}(\mathbb{B}) \cap L^{2,\omega}_{\Hq}(\mathbb{B}_I) $ 
as the space of all slice left slice regular functions $\varphi %\in \mathcal{SR}(\mathbb{B})
$
in $\mathbb{B}$ belonging to $L^{2,\omega}_{\Hq}(\mathbb{B}_I):=L^2\left( \mathbb{B}_I, \omega (|q|^2)  dxdy\right)$ and endowed with the norm induced from the slice inner product on $\mathbb{B}_I=\mathbb{B}\cap \C_I$,
$$\scal{f,g}_{\omega} =\int_{\mathbb{B}_I}\overline{f(x+Iy)} g(x+Iy)  \omega (x^2+y^2) dxdy .$$ 
More explicitly,  the Hilbert space $A^{2,\omega}_{slice}$ consists of all convergent power series $\varphi(q)=\sum_{n=0}^\infty q^n c_n$ on $\mathbb{B}$ for which the quaternionic sequence $(c_n)_n$ satisfies the growth condition
$$ \sum_{n=0}^\infty  \gamma_{n}  |c_n|^2  <\infty ; \quad \gamma_n  := \int_0^1 t^n \omega(t) dt . 
$$ 
The specification of the weight function $$\omega_{\beta,\eta}(t):=t^{\beta-1}(1-t)^{\eta-1}, \,\, \eta,\beta>0,$$
 gives rise to the weighted hyperholomorphic Hilbert space 
\begin{align} \label{SeqCharGrgSpace}
 A^{2,\beta,\eta}_{slice}:= \left\{ \varphi(q)=\sum_{n=0}^\infty q^n c_n; \, %\pi
 \sum_{n=0}^\infty  %\frac{\Gamma(\eta)\Gamma(\beta+n)}{\Gamma(\beta+\eta+n)}
 \gamma_n^{\beta,\eta}  |c_n|^2  <\infty  \right\}, 
 \end{align} 
 where 
 $$ \gamma_n^{\beta,\eta} : =   \frac{\Gamma(\eta)\Gamma(\beta+n)}{\Gamma(\beta+\eta+n)}.$$ 
It should be mentioned here that the monomials $e_n(q)= q^n$ form an orthogonal basis of $A^{2,\beta,\eta}_{slice}$ with square norm given by 
$$ \norm{e_n}_{\beta,\eta}^2 = \pi \gamma_n^{\beta,\eta}.$$ 
Moreover, appealing to the continuity of the evaluation linear form and the quaternionic version of Riesz representation theorem, we claim that $A^{2,\beta,\eta}_{slice}$ is a reproducing kernel Hilbert space, whose kernel function is expressible in terms of the quaternionic Gauss hypergeometric function (of first kind) 
\begin{align} \label{SliceGauss}
{_2F_1}^*\left(   \begin{array}{c}a,  b \\ c \end{array} \bigg | [p,q] \right) = \sum_{k=0}^\infty \frac{(a)_k (b)_k}{(c)_k} \frac{p^k q^k}{k!}
\end{align}
for $p,q\in \mathbb{B}$ and reals $a,b$ and $c$, where $(a)_k= a(a+1) \cdots (a+k-1)$ with $(a)_0=1$.
The series in \eqref{SliceGauss} converges absolutely and uniformly on $K\times K'$ for any compact subsets $K,K'\subset \mathbb{B}$. The function ${_2F_1}^*$ can be seen as the slice regularization of the classical Gauss hypergeometric function with respect to star product fo slice functions, in order to get a left slice regular function in $p$ and a right slice one in $q$. 
Namely, we assert

\begin{proposition}\label{kernel}
	The reproducing kernel of $A^{2,\beta,\eta}_{slice}$ is given by 
	\begin{align} \label{RepKerSlice} K_{\beta,\eta}(p,q) = 
	\frac{\Gamma(\beta+\eta)}{\pi \Gamma(\eta)\Gamma(\beta)}
	{_2F_1}^*\left(   \begin{array}{c}1,  \eta +\beta \\ \beta \end{array} \bigg | [p,\overline{q}] \right) 
	.
	\end{align}
\end{proposition}

\begin{proof} The explicit expression of $K_{\beta,\eta}(p,q)$ follows easily since 
	\begin{align}\label{RepKerExp}
	K_{\beta,\eta}(p,q)  &= \frac{1}{\pi}
	\sum_{n=0}^\infty \frac{e_n(p) \overline{e_n(q)}}{\gamma^{\beta,\eta}_n} \nonumber \\& = \frac{\Gamma(\beta+\eta)}{\pi \Gamma(\eta)\Gamma(\beta)}
	\sum_{n=0}^\infty \frac{(\beta+\eta)_n}{(\beta)_n} e_n(p) \overline{e_n(q)}.
	\end{align} 
\end{proof}

\begin{remark}
	For $\beta =1$, the space $A^{2,\eta,1}_{slice}$ is the one described in the introduction,
	% \eqref{HypBspace}, 
	 $A^{2,\eta,1}_{slice} = A^{2,\eta}_{slice}$. Moreover, the $K_{\eta,1}(p,q)$ reduces further to the reproducing kernel of classical weighted Bergman space $A^{2,\eta}_{slice}$ given by \cite[Theorem 3.1]{ElkachGh2018}
	\begin{align}
	K_{\eta,1}(p,q) =  \frac{\eta}{\pi} 
	{_1F_0}^*\left(   \begin{array}{c}   -\eta -1 \\ - \end{array} \bigg | [p,\overline{q}] \right) 
	\left( 1 - 2  \Re(q) \overline{p}   +   |q|^2\overline{p}^2  \right)^{-\eta-1}.
	\end{align}  
%\end{remark}
%
%\begin{remark}
	The restriction of $K_{\eta,1}$ to $ \mathbb{B}_I$ coincides with the classical Bergman kernel $ 
	K_{\eta,1}(z,w)  = (\eta/\pi)( 1 - z\overline{w} )^{-\eta-1}$; $z,w\in \mathbb{B}_I$.
\end{remark}

%%%%%%%%%%%%%%%%%%%%%%%%%%%%%%%%%%%%%%%%%%%%%%%%%
\section{Boundedness of the dual transforms $S^{\alpha}_{y}$}
%%%%%%%%%%%%%%%%%%%%%%%%%%%%%%%%%%%%%%%%%%%%%%%%%

We begin by noticing that the transform  $ S^{\alpha}_{y}$ satisfies 
$ S^{\alpha}_{y}\varphi (pq) 
%= \mathcal{L}_{pq}^\alpha(\varphi)(y)=\mathcal{L}_p^\alpha\circ\mathcal{L}^\alpha_q (\varphi)(y)
= S^{\alpha}_{y}\circ  \mathcal{L}^\alpha_q (\varphi)(p)$
by means of the semi-group property $\mathcal{L}_p^\alpha\circ\mathcal{L}^\alpha_q
=\mathcal{L}^\alpha_{pq} $ for the quaternionic fractional Hankel transform, as well as the eigenvalue equation
$ S^{\alpha}_{y}( \varphi_n^{\alpha} )   
= \varphi_n^{\alpha}  (y) e_n$ 
since the normalized generalized Laguerre polynomials 
\begin{equation} \label{basisLaguerre}
\varphi_n^{\alpha} (x) =\left( \frac{n!}{\Gamma(\alpha +n+1)}\right)^{1/2} L^{(\alpha)}_{n}(x)
\end{equation}
are solutions of $\mathcal{L}_q^\alpha(L ^{(\alpha)}_{n}) =  q^{n}L^{(\alpha)}_{n}$. 
 Moreover, the kernel function 
 \begin{align}\label{KernelBessel} 
 R_q^\alpha(x,y) &=
 \frac{1}{(1-q) \sqrt{q xy}^{\alpha}}
 \exp\left(-\frac{q(x+y)}{1-q}\right) 
 I_\alpha\left( \frac{2\sqrt{qxy}}{1-q} \right),
 \end{align}
 %for $x\in (0,+\infty)$, $y>0$ and $\alpha>-1$, 
 for the transform $S^{\alpha}_{y}$ in \eqref{SIntT},
 has the expansion series 
 \cite{Namias1980b,ElkGhHa20},
 \begin{align} \label{expKerRHH}
 R_p^\alpha(x,y)= \sum_{n=0}^\infty e_n(p) \varphi_n^{\alpha} (x) \varphi_n^{\alpha} (y) 
 \end{align} 
 which follows from the Hille--Hardy formula for the Laguerre polynomials \cite[(6.2.25) p. 288]{AndrewsAskeyRoy1999}. 
Such kernel function satisfies the following reproducing property.

\begin{proposition}\label{KerKer} Let $K_{\beta,\eta}(p,q) $ be as in \eqref{RepKerSlice}. Then, for every $y\in (0,+\infty)$, we have 
	\begin{align}\label{KernelRern}  	R_q^\alpha(x,y)  = \int_{\mathbb{B}_I} \overline{K_{\beta,\eta}(p,q)   }  R_p^\alpha(x,y) \omega_{\beta,\eta}(|p|^2) d\lambda_I(p) 
	\end{align}
	and 
	\begin{align} \label{EqWD} 
		R_{|q|^2}^\alpha(y,y) =  \int_{\R^+} |R_q^\alpha(x,y) |^2 x^\alpha e^{-x} dx  .
	\end{align} 
\end{proposition}

\begin{proof}   Both \eqref{KernelRern} and \eqref{EqWD} can be proved, at least formally,  using the expansion series of the involved kernels given by \eqref{RepKerExp} and \eqref{expKerRHH}.  
\end{proof}

%\begin{remark}
	Proposition \ref{KerKer} can be used to reprove the reproducing 
	property satisfied by the functions in the range of $S^{\alpha}_{y}$ by means of the kernel $K_{\beta,\eta}$. Indeed, by
 rewriting $S^{\alpha}_{y}$ as 
	$ %\begin{align}\label{KernelRepS} 
	S^{\alpha}_{y}\varphi (q) = \scal{\overline{R_q^\alpha(\cdot, y)},\varphi} _{L^{2,\alpha}_{\Hq}(\R^+)} 
	$,  inserting \eqref{KernelRern}
	%\eqref{KernelRepS}
	 and making use of Fubini theorem we obtain
	\begin{align*}
	S^{\alpha}_{y} ( \varphi)(q) 	
	&= \scal{  \scal{   K_{\bullet}^\alpha(\cdot,y) , K_{\beta,\eta}(\bullet ,q) }_{A^{2,\omega}_{slice}}  ,\varphi (\cdot)}_{L^{2,\alpha}_{\Hq}(\R^+)}
	\\&= \scal{ K_{\beta,\eta}(\bullet ,q), \scal{  \overline{ K_{\bullet}^\alpha(\cdot,y)   } ,\varphi(\cdot) }_{L^{2,\alpha}_{\Hq}(\R^+)} }_{A^{2,\omega}_{slice}}
	\\&= \scal{ K_{\beta,\eta}(\bullet ,q), S^{\alpha}_{y} ( \varphi)(\bullet) }_{A^{2,\omega}_{slice}}
	\end{align*}
	for every  $\varphi \in  L^{2,\alpha}_{\Hq}(\R^+)$. 	 %\end{remark}
Furthermore, 
we get easily
\begin{align}  
|S^{\alpha}_{y}\varphi (q)|^2  &\leq  \scal{R_{q}^\alpha(\cdot,y), R_{q}^\alpha(\cdot,y)}_{L^{2,\alpha}_{\Hq}(\R^+)}  \norm{\varphi}_{L^{2,\alpha}_{\Hq}(\R^+)}^2 \nonumber
\leq R_{|q|^2}^\alpha(y,y) \norm{\varphi}_{L^{2,\alpha}_{\Hq}(\R^+)}^2 \label{ineqWD}
\end{align}  
by means of Cauchy-Schwarz inequality and identity \eqref{EqWD}. This proves that the transform $S^{\alpha}_{y}$ is well defined on $L^{2,\alpha}_{\Hq}(\R^+)$.
In addition, we have 
\begin{align*}  
\int_{\mathbb{B}_I} |S^{\alpha}_{y}\varphi (q)|^2 \omega (|q|^2)  d\lambda_I(q) 
\leq \left( \int_{\mathbb{B}_I} R_{|q|^2}^\alpha(y,y) \omega (|q|^2) d\lambda_I(q)\right) \norm{\varphi}_{L^{2,\alpha}_{\Hq}(\R^+)}^2 .
\end{align*} 
Accordingly, under the assumption that 
\begin{equation}\label{BoundeCond} 
\int_{\mathbb{B}_I} R_{|q|^2}^\alpha(y,y)  \omega (|q|^2)  dudv <+\infty; \quad q=u+Iv,
\end{equation}
the transform  $S^{\alpha}_{y}$ is a bounded operator from $ L^{2,\alpha}_{\Hq}(\R^+)$ into $L^{2,\omega}_{\Hq}(\mathbb{B}_I)$. 
For $y=0$, the assumption that 
\eqref{BoundeCond} reduces further to 
\begin{equation}\label{BoundeCondy0} 
\int_0^1  R_{t}^\alpha(0,0)  \omega (t)  dt = 
\frac{1}{2^\alpha \Gamma(\alpha+1)} \int_0^1  \frac{\omega (t)}{(1-t)^{\alpha+1}}    dt
\end{equation}
be finite. The convergence of the integral in \eqref{BoundeCondy0}  
readily holds when $\eta > \alpha+1$ for the special case of $\omega(t) =\omega_{\beta,\eta}(t):=t^{\beta-1}(1-t)^{\eta-1}$ with $\alpha>-1$ and $\beta,\eta>0$. 
The next result extends this condition to includes $y>0$.

\begin{proposition}\label{thmBound}
	Let  $\alpha>-1$, $\beta,\eta>0$ and $y\geq 0$. It in addition $\eta >\alpha +1$, then the integral operator $S^{\alpha}_{y}: L^{2,\alpha}_{\Hq}(\R^+) \longrightarrow L^{2,\omega_{\beta,\eta}}_{\Hq}(\mathbb{B}_I)$ is bounded.
\end{proposition}

\begin{proof}
	Let denote the quantity in \eqref{BoundeCond} by $\ell_\alpha^{\beta,\eta}$
	for $\omega =\omega_{\beta,\eta}$. Then, we have 
	\begin{align*} 
	\ell_\alpha^{\beta,\eta} 
	&= \pi\int_0^1  R_{t}^\alpha(y,y)  \omega_{\beta,\eta} (t)  dt\\
	&=\pi\int_0^1 t^{\beta-1} (1-t)^{\eta-\alpha-2} i_\alpha\left( \frac{2y\sqrt{t}}{1-t}\right) \exp\left( -\frac{2yt}{1-t}\right) dt	
	\end{align*}
	where $i_\alpha(x) := (x/2)^{-\alpha}I_\alpha(x)$ is nonnegative and bounded on $\R^+$ by some constant $c_\alpha$. Thus, we have 
	\begin{align*} 
\ell_\alpha^{\beta,\eta} 
  &\leq c_\alpha \int_0^{1} t^{\beta-1} (1-t)^{\eta-\alpha-2}   \exp\left( -\frac{2yt}{1-t}\right) dt 
	\\& \leq c_\alpha 
	\int_0^{+\infty} u^{\beta-1} (1+u)^{\alpha-\beta-\eta+1}  \exp\left( -2y u\right) du. 
\end{align*}
The last integral follows making the change of variable $u= t/(1-t)$. It is clearly convergent when $\eta>\alpha+1$ and $\beta >0$. %Thus, we recover the condition provided when $y=0$. 
\end{proof}

The next result refines the  boundedness condition of $S^{\alpha}_{y}$ provided in the previous assertion. It shows that  $\eta > \alpha+1$ can be relaxed. To this end, we distinguish 
two cases $y=0$ and $y>0$.

\begin{proposition}\label{thmBound}
	Let  $\alpha>-1$, $\beta,\eta>0$ and $y\geq 0$. Then, the integral operator $S^{\alpha}_{y}: L^{2,\alpha}_{\Hq}(\R^+) \longrightarrow L^{2,\omega_{\beta,\eta}}_{\Hq}(\mathbb{B}_I)$ is bounded for any $y>0$. The boundedness of $S^{\alpha}_{y}$ at $y=0$ holds when $\eta\geq \alpha$.
%	
%	 in the following cases
%	\begin{itemize}
%		\item[i)]   $\eta\geq \alpha$ when $y=0$. 
%		\item[ii)]  $\eta\geq -1/2$ when $y>0$. 
%	\end{itemize}
\end{proposition}

\begin{proof}
	Set 
	$$c^{\alpha,\beta,\eta}_n(y) := \pi   \gamma_n^{\beta,\eta}\left| \varphi_n^{\alpha}(y)\right|^2 %=\norm{S^{\alpha}_{y} \varphi_n^{\alpha} }_{\omega_{\beta,\eta}}^2 
	=  \pi \frac{\Gamma(\eta)\Gamma(\beta+n)}{\Gamma(\beta+\eta+n)} 
	\left| \varphi_n^{\alpha}(y)\right|^2
	. $$
	Then, 
%	$$\sup_{n} c^{\alpha,\beta,\eta}_n(y)=\sup_{n}\left( \norm{S^{\alpha}_{y} \varphi_n^{\alpha} }_{\omega_{\beta,\eta}}^2\right)  \leq \sup_{\varphi\in L^{2,\alpha}_{\Hq}(\R^+)}\left( \norm{S^{\alpha}_{y} \varphi  }_{\omega_{\beta,\eta}}^2\right).$$ 
%	On the other hand, 
	for every $\varphi = \sum_{n=0}^\infty  a_n \varphi_n^{\alpha} \in L^{2,\alpha}_{\Hq}(\R^+)$, we have $\norm{\varphi}_{L^{2,\alpha}_{\Hq}(\R^+)}^2 = \sum_{n=0}^\infty   |a_n|^2  <+\infty$ and by means of \eqref{NormSalpha} we get
	\begin{align*}
	\norm{S^{\alpha}_{y} \varphi}_{\omega_{\beta,\eta}}^2 
	\leq \pi \sup_{n}\left( \gamma_n^{\beta,\eta} \left| \varphi_n^{\alpha}(y)\right|^2\right)  	\sum_{n=0}^\infty  \left|a_n\right| ^2
	\leq \sup_{n}\left(  c^{\alpha,\beta,\eta}_n(y)\right)  \norm{\varphi}_{L^{2,\alpha}_{\Hq}(\R^+)}^2.
	\end{align*}
	Subsequently, $ S^{\alpha}_{y} $ is bounded, as operator from $ L^{2,\alpha}_{\Hq}(\R^+)$ into $L^{2,\omega_{\beta,\eta}}_{\Hq}(\mathbb{B}_I)$, if $\sup_{n} \left( c^{\alpha,\beta,\eta}_n(y)\right) $ is finite. 
For the special case of $y=0$ we have  
	\begin{align*}
	c^{\alpha,\beta,\eta}_n(0) &=
	\frac{\pi\Gamma(\eta)}{\Gamma^2(\alpha+1)}  \frac{\Gamma(\alpha+n+1)}{\Gamma(n+1)}
	\frac{ \Gamma(\beta+n)}{\Gamma(\beta+\eta+n)} 
	\sim \frac{\pi\Gamma(\eta)}{\Gamma^2(\alpha+1)} 
	n^{\alpha-\eta}
	\end{align*} 
	for $n$ large enough, and therefore, its supermum is finite if and only if $\eta\geq\alpha$. Thus,  $S^{\alpha}_{0}$ is bounded when $\eta\geq \alpha$. 
		
	For  arbitrary fixed $y>0$, there exists  some  positive constant $ M^{\alpha,\eta}(y)$,  depending only in $\alpha,\eta$ and $y$, such that  
		\begin{align}\label{asymcny}
	c^{\alpha,\beta,\eta}_n(y)
	\leq   M^{\alpha,\eta}(y)     
	n^{-\eta-1/2}  
	\end{align}
	holds true for large $n$. This follows making use of 
	 the asymptotic behavior for gamma function as well as the one for generalized Laguerre polynomials 
	\cite[p.245]{MagnusOberhettingerSoni1966} 
		\begin{align}\label{asymcnLaguerre}
	L^{(\alpha)}_n(y) 
	&=    \frac{e^{x/2}}{\sqrt{\pi} x^{(2\alpha+1)/4}} 
	n^{(2\alpha-1)/4} \cos\left( 2 \sqrt{n y} - \pi \frac{2\alpha+1}4 \right) 
	+ O \left(n^{(2\alpha-3)/4}  \right) .
	\end{align}  
	Therefore, the quantity $\sup_{n} ( c^{\alpha,\beta,\eta}_n(y)) $  is clearly finite if  $\eta\geq -1/2$ is assumed which is satisfied since $\eta >0$.  
\end{proof}

%\begin{remark}
%	The boundedness assertion \ref{thmBound} remains valid for 
%	$ S^{\alpha}_{y}: L^{2,\alpha}_{\Hq}(\R^+) \longrightarrow {^cA}^{2,\omega}_{slice}$ when $y \in\mathcal{Z}\{L_n^\alpha;\, n\}$.
%\end{remark}

%%%%%%%%%%%%%%%%%%%%%%%%%%%%%%%%%%%%%%%%%%%%%%%%%
\section{The null space and the range of  $S^{\alpha}_{y}$}
%%%%%%%%%%%%%%%%%%%%%%%%%%%%%%%%%%%%%%%%%%%%%%%%%

Apparently, the description of the null space and the range of $S^{\alpha}_{y}$ depends on the set $ \mathcal{Z} \{L^{(\alpha)}_{n};\, n\}:=  \cup_n \mathcal{Z} (L^{(\alpha)}_{n})$, where $\mathcal{Z} (L^{(\alpha)}_{n})$ denotes the zero set of $L^{(\alpha)}_{n}$.

\begin{proposition}
	The null space of $S^{\alpha}_{y}$ in $L^{2,\alpha}_{\Hq}(\R^+)$ is  spanned by $\varphi_n^{\alpha}$ with $n\in N_y^\alpha =\{ n;  L^{(\alpha)}_{n}(y)\ne 0\}$,  
	$ \ker(S^{\alpha}_{y})=  span\{\varphi_n^{\alpha} ; \, n\in N_y^\alpha  \}$.
\end{proposition} 

\begin{proof}
	It is clear that $ span\{\varphi_n^{\alpha} ; \, n\in N_y^\alpha  \} \subset \ker(S^{\alpha}_{y}) $, since $S^{\alpha}_{y} \varphi_n^{\alpha}  =0$ for any  $n\in N_y^\alpha$. 
	Conversely,  let $\varphi \in \ker(S^{\alpha}_{y}) = \{ \varphi \in L^{2,\alpha}_{\Hq}(\R^+) ; S^{\alpha}_{y}(\varphi)=0\}$ that we can  expanded as $ \varphi=\sum_{n=0}^\infty a_n \varphi_n^{\alpha} $, since the Laguerre functions in \eqref{basisLaguerre} 
	form an orthonormal basis of $L^{2,\alpha}_{\Hq}(\R^+)$. Then, keeping in mind that $S^{\alpha}_{y} \varphi_n  = \varphi_n^{\alpha}(y)e_n $, we get  
	$ S^{\alpha}_{y}(\varphi)= \sum_{n=0}^\infty a_n  \varphi_n^{\alpha}  (y)\varphi_n^{\alpha} $
	and hence
	$0=\scal{  S^{\alpha}_{y}(\varphi), e_k} = \pi \overline{a_k}   \varphi_k^{\alpha}  (y) \gamma_k$.
	Hence, $a_k =0$ for any $k\notin N_y^\alpha$ and therefore $\varphi =\sum_{n\in N_y^\alpha} a_n \varphi_n^{\alpha}  $.
	Moreover, the dimension of $\ker(S^{\alpha}_{y})$ is clearly given by $\dim(\ker(S^{\alpha}_{y}))= card(N_y^\alpha)$.
\end{proof}

\begin{remark}
	We notice that $\dim(\ker(S^{\alpha}_{y}))$ depends in $y$ and $\alpha$ and characterizes the number (finite or infinite) of generalized Laguerre polynomials that have $y$ as common zero. Thus, for $y=0$, the set $N_0$ is empty since $L^{(\alpha)}_{n}(0)\ne 0$ for any nonnegative integer $n$, so that $\dim(\ker(S^{\alpha}_{y}))=0$. 
	By regarding the graphs of the generalized Laguerre polynomials we conjecture that $card(N_y^\alpha)$ (and hence  $\dim(\ker(S^{\alpha}_{y}))$) is finite. 
\end{remark}

The next result shows that the Hilbert space $A^{2,\omega}_{slice}$ shelters the range ${^cA}^{2,\omega}_{slice}: = S^{\alpha}_{y} (L^{2,\alpha}_{\Hq}(\R^+))$ of $S^{\alpha}_{y}$ acting on $L^{2,\alpha}_{\Hq}(\R^+)$. 

\begin{proposition}\label{range} Assume that \eqref{BoundeCond} holds, then for every $y\geq 0$  we have 
	$ {^cA}^{2,\omega}_{slice}   \subset A^{2,\omega}_{slice}$
	and $\{e_n(q)=q^n; n\notin N_y^\alpha\}$  defines a complete orthogonal system in 
	${^cA}^{2,\omega}_{slice}$.
\end{proposition}

\begin{proof}
	Starting from $ S^{\alpha}_{y}(\varphi)= \sum_{n=0}^\infty a_n  \varphi_n^{\alpha}  (y)\varphi_n^{\alpha} $ for given $\varphi \in L^{2,\alpha}_{\Hq}(\R^+)$ and using the fact that 
	$\scal{e_m,e_n}_{\omega}=   \pi \gamma_{n} \delta_{m,n}$, we get easily
	\begin{align} \label{NormSalpha}
	\norm{S^{\alpha}_{y} \varphi}_{\omega}^2  
	= 	\pi \sum_{n=0}^\infty  \gamma_{n} \left| \varphi_n^{\alpha}(y)  a_n\right| ^2 <+\infty  
	\end{align}
	under the assumption that $S^{\alpha}_{y}$ is bounded. Therefore,
	\begin{align} \label{Inclusion}
	S^{\alpha}_{y} (L^{2,\alpha}_{\Hq}(\R^+))  \subset  \left\{ \sum_{n=0}^\infty q^n c_n; \, c_n \in \Hq ; q\in \mathbb{B}, \, \sum_{n=0}^\infty  \gamma_{n} |c_n|^2  <\infty  \right\} .
	\end{align}
	The right hand-side in \eqref{Inclusion} is exactly the sequential characterization of the weighted hyperholomorphic Bergman space  $A^{2,\omega}_{slice}$ discussed in the  Section 2.
\end{proof}

\begin{remark}\label{RemIncomp}
	If $y$  is a positive zero of some Laguerre polynomial, then $N_y^\alpha$ is not empty and the corresponding monomials $e_n$; $n\in  N_y^\alpha$, do not belong to ${^cA}^{2,\omega}_{slice}$. This shows that, in this case, 
	${^cA}^{2,\omega}_{slice}$ is strictly contained in  $A^{2,\omega}_{slice}$. 
\end{remark}

\begin{remark}\label{Remy0}
	For $y=0$, we have $\varphi_n^{\alpha}(0)\ne 0$
%	$$\varphi_n^{\alpha}(0) = \frac{1}{ \Gamma(\alpha+1)} \left(\frac{\Gamma(\alpha+n+1)}{n!} \right)^{1/2} $$
and then  $N_0^\alpha$ is an empty set. Thus, we can show that for $\omega=\omega_{\beta,\eta}$ with $\beta=1$ and $\eta=\alpha$, we have 
	$ {^cA}^{2,\omega}_{slice}   =  A^{2,\beta,\eta}_{slice}=A^{2,\alpha}_{slice}$
and hence $S^{\alpha}_{0}: L^{2,\alpha}_{\Hq}(\R^+) \longrightarrow  A^{2,\beta,\eta}_{slice}$ is onto and is exactly the second Bargmann transform $\SHyperBTransR$ in \eqref{BargmannT}
	defining a unitary  isometric transformation from  $L^{2,\alpha}_{\Hq}(\R^+)$
	onto the slice hyperholomorphic Bergman space 
	$A^{2,\alpha}_{slice}$.
\end{remark}

\begin{remark}\label{Weightassump}
	For the general case; the converse inclusion in \eqref{Inclusion} requires further assumption on the weight function.
\end{remark}

%%%%%%%%%%%%%%%%%%%%%%%%%%%%%%%%%%%%%%%%%%%%%%%%%
\section{Compactness and membership in $p$-Schatten class}
%%%%%%%%%%%%%%%%%%%%%%%%%%%%%%%%%%%%%%%%%%%%%%%%%

\begin{proposition}\label{corcomp}
	The operator
	$ S^{\alpha}_{y} :L^{2,\alpha}_{\Hq}(\R^+) \longrightarrow A^{2,\beta,\eta}_{slice}$ is compact for any $y>0$. 
\end{proposition}

\begin{proof}
	The compactness of $ S^{\alpha}_{y} $; $y>0$, follows by appealing to the spectral theorem  %(for non-self-adjoint) operators
	 \cite[Theorem 4.3.5]{HsingEubank2015}. In fact, $ S^{\alpha}_{y} $ is bounded and can be expanded as 
	\begin{align*}
	S^{\alpha}_{y} \varphi 
	&= \sum_{n=0}^\infty s_{n,y}^{\alpha} \frac{e_n }{\sqrt{\pi\gamma_n} } \scal{\varphi , \varphi_n^{\alpha} }_{L^{2,\alpha}_{\Hq}(\R^+)} 
	\end{align*}
	with $(\varphi_n^{\alpha} )_n$ and $\left( {e_n }/{\sqrt{\pi\gamma_n} }\right)_n$ are orthonormal bases of $L^{2,\alpha}_{\Hq}(\R^+)$ and $A^{2,\omega}_{slice}$, respectively, and	
	$s_{n,y}^{\alpha}:= \sqrt{\pi\gamma_n}\varphi_n^{\alpha}(y)$
	tends to $0$ when the wight function $\omega_{\beta,\eta}$ is specified. This readily follows making use of \eqref{asymcnLaguerre} 
	thanks to \eqref{asymcny} since 
	$$ |s_{n,y}^{\alpha}| =   \sqrt{c^{\alpha,\beta,\eta}_n(y)} .$$
\end{proof}

The membership of $ S^{\alpha}_{y} $ in the $p$-Schatten class is a direct  consequence of the next two results. Recall for instance that a bounded operator $S$ is said to be a Schatten operator of class  $p$ for $p\geq 1$ if 
its Schatten $p$-norm 
$$
\norm{S}_{p}:= \mbox{Tr} (|S|^{p})^{1/p} $$
is finite.

\begin{lemma} \label{adjointS}
	The adjoint of $S^{\alpha}_{y}$ is given by 
	\begin{equation}\label{adjointSexp}
	(S^{\alpha}_{y})^*  G  (x) = \int_{\mathbb{B}_I} K_{\overline{q} }^\alpha(x,y)    G(q) (1-|q|^2)^{\alpha-1}d\lambda_I(q).
	\end{equation}
\end{lemma}

\begin{proof}
	For every $\varphi  \in L^{2,\alpha}_{\Hq}(\R^+) $ and $ G \in A^{2,\omega}_{slice}$ 
	we have
	\begin{align*}
	\scal{ S^{\alpha}_{y}\varphi,G}_{A^{2,\omega}_{slice}} 
	& =    \sum_{n=0}^\infty \varphi_n^{\alpha} (y) \scal{\varphi , \varphi_n^{\alpha} }_{L^{2,\alpha}_{\Hq}(\R^+) }  \scal{ e_n ,  G }_{A^{2,\omega}_{slice}}
	=   \scal{\varphi , (S^{\alpha}_{y})^*  G}_{L^{2,\alpha}_{\Hq}(\R^+) }.
	\end{align*}
	This readily follows by applying Fubini's theorem. 
	Therefore, the adjoint of $S^{\alpha}_{y}$ given by
	\begin{align*}
	(S^{\alpha}_{y})^*  G(x)  =   \scal{ \sum_{n=0}^\infty \varphi_n^{\alpha} (y)  \varphi_n^{\alpha}   e_n  ,  G }_{A^{2,\omega}_{slice}}  =    \scal{  K(x,y| \cdot ) ,  G  }_{A^{2,\omega}_{slice}} 
	\end{align*}
which	reduces further to 
	\eqref{adjointSexp}
	since $K(x,y| q)$ is exactly the kernel function in \eqref{KernelBessel}. 	
\end{proof}

\begin{proposition} \label{propsv}
	If $  S^{\alpha}_{y}$  is bounded,  then their singular values are given by \begin{align}\label{ingeigenvS}
	|s_{n,y}^{\alpha}| = \left( \pi \gamma_n\right) ^{1/2} |\varphi_n^{\alpha} (y)|.
	\end{align}
\end{proposition}

\begin{proof}	
	By Lemma \ref{adjointS}, we have $	(S^{\alpha}_{y})^*  e_k =	\pi \gamma_k  \varphi^\alpha_k(y) \varphi^\alpha_k $, and therefore, the operator $(S^{\alpha}_{y})^*   S^{\alpha}_{z} : L^{2,\alpha}_{\Hq}(\R^+) \longrightarrow L^{2,\alpha}_{\Hq}(\R^+)$ satisfies
	$$ 
	(S^{\alpha}_{y})^*   S^{\alpha}_{z} \varphi_n^{\alpha}  
	=  (S^{\alpha}_{y})^* \left(  \varphi_n^{\alpha} (z) e_n\right)  
	=  \pi \gamma_n  \varphi_n^{\alpha} (y) \varphi_n^{\alpha} (z) \varphi_n^{\alpha} .$$
	This is to say that the Laguerre functions 
	$\varphi^{\alpha}_{n}$ in \eqref{basisLaguerre} 
	constitute an orthogonal basis of $L^2$-eigenfunctions for $(S^{\alpha}_{y})^*   S^{\alpha}_{y}$ with $\pi \gamma_n (\varphi_n^{\alpha} (y) )^2$ as corresponding eigenvalues. Therefore, the eigenvalues of $|(S^{\alpha}_{y})^*|:= ( (S^{\alpha}_{y})^*S^{\alpha}_{y})^{1/2}$ are exactly those given through \eqref{ingeigenvS}.
\end{proof}

\begin{proposition}\label{thmSchayyen}
	Let $y>0$ and $p>4/(1+2\eta)$. Then, 
	$ S^{\alpha}_{y}  :  L^{2,\alpha}_{\Hq}(\R^+) \longrightarrow A^{2,\beta,\eta}_{slice}$ is a Schatten operator of class $p$. 
\end{proposition}

\begin{proof}
	The operator $S^{\alpha}_{y}$; $y>0$, is compact thanks to Proposition \ref{corcomp}. To conclude we need only to apply Proposition  \ref{propsv} keeping in mind \eqref{asymcny}
	 for large $n$. 
	 Thus, the singular values $s_{n,y}^{\alpha}$ satisfy 
	\begin{align}
	|s_{n,y}^{\alpha}|=\sqrt{ 
		c^{\alpha,\beta,\eta}_n(y)}
	= O(n^{-(2\eta+1)/4}).
%	&\sim  m^{\alpha,\eta}(y)     	n^{-(2\eta+1)/4} \left |\cos\left( 2 \sqrt{n y} - \pi \frac{2\alpha+1}4 \right)\right|   
	\end{align}
	Therefore, if $(1+2\eta)p>4$, the the series $\sum_{n=0}^\infty  |s_{n,y}^{\alpha}| ^p$ converges and therefore 
	$$\norm{S^{\alpha}_{y}}_{p}:= \mbox{Tr} (|S^{\alpha}_{y}|^{p})^{1/p} = 
\left( 	\sum_{n=0}^\infty  |s_{n,y}^{\alpha}| ^p \right)^{1/p} <+\infty.$$
The second equality follows since $S^{\alpha}_{y}$ is compact from/into separable Hilbert spaces. 
\end{proof}

\begin{remark}
	For $y=0$ and large $n$, we have 
	\begin{align*} |s^\alpha_n(0)|=\sqrt{c^{\alpha,\beta,\eta}_n(0)}&
	\sim \frac{\sqrt{\pi\Gamma(\eta)}}{\Gamma(\alpha+1)}n^{(\alpha-\eta)/2}  .
	\end{align*}
	Subsequently, the transform $S^{\alpha}_{0}:  L^{2,\alpha}_{\Hq}(\R^+) \longrightarrow A^{2,\beta,\eta}_{slice}$ is compact if and only if $\eta>\alpha$, since in this case $\lim\limits_{n \to +\infty} |s^\alpha_n(0)|=0$. Moreover, it is in $p$-Schatten class if in addition $p>2/(\eta-\alpha)$.
\end{remark}

We conclude by providing the singular value decomposition of $S^{\alpha}_{0}$. Thus, we consider the mapping defined by 
$U^{\alpha}_{y}(\varphi_n^{\alpha} ) =  0$ for $n\in N_y^\alpha$ and   
$$U^{\alpha}_{y}(\varphi_n^{\alpha} ) =  
\frac{\varphi_n^{\alpha} (y)}{\sqrt{\pi \gamma_n} |\varphi_n^{\alpha} (y)|} e_n 
%; \quad n\notin N_y^\alpha, 
$$
otherwise. We extend $U^{\alpha}_{y}$ in a natural way to  linear mapping on the whole $ L^{2,\alpha}_{\Hq}(\R^+)$. 
By means of  $ \norm{ e_n}_{L^{2,\omega}_{\Hq}(\mathbb{B}_I)}^2 = \pi \gamma_n $, we get 
$$\norm{U^{\alpha}_{y}(\varphi_n^{\alpha} )}_{L^{2,\alpha}_{\Hq}(\R^+)} =
\left|\frac{\varphi_n^{\alpha} (y)}{\sqrt{\pi \gamma_n} |\varphi_n^{\alpha} (y)|} \right| \norm{ e_n}_{L^{2,\omega}_{\Hq}(\mathbb{B}_I)} = 1. $$ 
 Then, we claim that the following assertions hold trues:

%\begin{proposition} The linear operator $U^{\alpha}_{y}$ defined on   $ L^{2,\alpha}_{\Hq}(\R^+)$ satisfies 
\begin{enumerate}
	\item[i)] $U^{\alpha}_{y} :  L^{2,\alpha}_{\Hq}(\R^+) \longrightarrow L^{2,\omega}_{\Hq}(\mathbb{B}_I) $ is a partial isometry.
	
	\item[ii)] We have $ \ker(S^{\alpha}_{y}) = \ker(U^{\alpha}_{y}) = span\{e_n; \, n\in N_y^\alpha  \}$.	
\end{enumerate}
%\end{proposition} 
Moreover,  the singular value decomposition of the linear operator $S^{\alpha}_{y}$ is given by 
$$S^{\alpha}_{y} = U^{\alpha}_{y} |S^{\alpha}_{y} |.$$  	This readily follows by direct computation.

%\begin{proposition}
%\end{proposition} 

\bibliographystyle{amsplain}

\begin{thebibliography}{99}
	
	\bibitem{AndrewsAskeyRoy1999} Andrews G.E., Askey R., Roy R.,
	{ Special functions}.
	Encyclopedia of Mathematics and its Applications, 71. Cambridge University Press: Cambridge; 1999.
	
	\bibitem{Bracewell1999} Bracewell R., The Hankel transform. The Fourier transform and its applications, 3rd ed. New York: McGraw-Hill, (1999) 244-250. 
	
	\bibitem{ElkachGh2018} Elkachkouri A.,  Ghanmi A., The hyperholomorphic Bergman space on $\mathbb{ B}_{R}$: Integral representation and asymptotic behavior.  
	Complex Anal. Oper. Theory 12 (2018), no. 5, 1351--1367.
	
	\bibitem{ElkGhHa20}  Elkachkouri A.,  Ghanmi A.,  Hafoud A.,
	Bargmann's versus of the quaternionic fractional Hankel transform.
	Submitted.
	
	\bibitem{HsingEubank2015} Hsing T., Eubank R.,
	Theoretical foundations of functional data analysis, with an introduction to linear operators 2015.
	
	\bibitem{Karp1994} Karp D.B.,
	Fractional Hankel transformation and its applications in mathematical physics. (Russian); translated from Dokl. Akad. Nauk 338 (1994), no. 1, 10--14 Russian Acad. Sci. Dokl. Math. 50 (1995), no. 2, 179--185
	
	\bibitem{Kerr1991} Kerr F.H., 
	A fractional power theory for Hankel transforms. 
	J. Math. Anal. Appl. 158 (1991) 114--123
	
	\bibitem{Krenk1982} Krenk S., Some integral relations of Hankel transform type and applications to elasticity theory. Integral Equations and Operator Theory, 5(1), (1982) 548--561. 
	
	\bibitem{MagnusOberhettingerSoni1966}   Magnus W., Oberhettinger F., Soni R.P.,
	Formulas and Theorems in the Special Functions of Mathematical Physics. 
	Springer-Verlag, Berlin, 1966.
	
	\bibitem{Namias1980b} Namias V., Fractionalization of Hankel transforms. J. Inst. Math. Appl. 26 (1980), no. 2, 187--197.
	
	\bibitem{PrasadMahatoSinghDixit2013} Prasad A., Mahato A., Singh V.K., Dixit M.M., The continuous fractional Bessel wavelet transformation. Bound. Value Probl. 2013, 40 (2013), 16 pp.
	
	
	
	\bibitem{SheppardLarkin1998} Sheppard C.J.R., Larkin, K.G., Similarity theorems for fractional Fourier transforms and fractional Hankel transforms. Opt. Commun. 154 (1998)  173--178.
	
	\bibitem{Uzun2015}  \"Unalmı\c{s} Uzun B., Fractional Hankel and Bessel wavelet transforms of almost periodic signals.  
	J. Inequal. Appl., 388 (2015), 12 pp.
\end{thebibliography}

\end{document}